\documentclass{amsart}

\usepackage{amsmath,amssymb,amsthm,amsfonts,verbatim,color,enumerate}
\usepackage{microtype}
\usepackage[all,2cell]{xy}
\usepackage{mathtools}
\usepackage{graphicx}
\usepackage{hyperref}
\usepackage{mathrsfs,pinlabel}

\CompileMatrices

\usepackage[top=1.3in,bottom=1.9in,left=1.4in,right=1.4in]{geometry}

\newtheorem{thm}{Theorem}
\newtheorem{lem}[thm]{Lemma}

\theoremstyle{definition} 

\newtheorem*{qu*}{Question}
 
\theoremstyle{remark} 
\newtheorem{rmk}[thm]{Remark}
\numberwithin{equation}{section} 

\newenvironment{defn}[1][Definition.]{\begin{trivlist}
\item[\hskip \labelsep {\bfseries #1}]}{\end{trivlist}}

\newcommand{\z}{\mathbb{Z}}

\newcommand{\ga}{\gamma}

\newcommand{\Ga}{\Gamma}
\newcommand{\what}{\widehat}\newcommand{\wtil}{\widetilde}

\newcommand{\cd}{\cdots}\newcommand{\ld}{\ldots}
\newcommand{\sbs}{\subset}\newcommand{\bs}{\backslash}\newcommand{\pa}{\partial}\newcommand{\es}{\emptyset}

\newcommand{\xra}{\xrightarrow}

\newcommand{\ra}{\rightarrow}
\newcommand{\hra}{\hookrightarrow}

\newcommand{\bb}[1]{\mathbb{#1}}

\newcommand{\ov}[1]{\overline{#1}}

\newcommand{\scr}{\mathscr}

\newcommand{\lan}{\langle}\newcommand{\ran}{\rangle}
\newcommand{\til}{\tilde}
\newcommand{\ti}{\times}

\newcommand{\rest}[2]{#1\bigr\vert_{#2}}\newcommand{\hy}{\bb H}\newcommand{\SO}{\text{SO}}\newcommand{\Lk}{\text{Lk}}

\usepackage{subfig}
\addtocontents{toc}{\protect\setcounter{tocdepth}{1}}

\begin{document}

\title{Hyperbolic groups with boundary an $n$-dimensional Sierpinski space}

\author{Jean-Fran\c{c}ois Lafont}
\address{Department of Mathematics, Ohio State University, Columbus, OH 43210}
\email{jlafont@math.ohio-state.edu}

\author{Bena Tshishiku}
\address{Department of Mathematics, University of Chicago, Chicago, IL 60615} 
\email{tshishikub@math.uchicago.edu}


\date{\today}


\begin{abstract}
For $n\ge7$, we show that if $G$ is a torsion-free hyperbolic group whose visual boundary $\pa_\infty G\simeq\scr{S}^{n-2}$ is an $(n-2)$-dimensional Sierpinski space, then $G=\pi_1(W)$ for some aspherical $n$-manifold $W$ with nonempty boundary. Concerning the converse, we construct, for each $n\geq 4$, examples of aspherical manifolds with boundary, whose fundamental group $G$ is hyperbolic, but with visual boundary $\pa_\infty G$ {\it not} homeomorphic to
$\scr{S}^{n-2}$. 
\end{abstract}

\maketitle

\section{Introduction}

One of the basic invariants for a hyperbolic group is its boundary at infinity, and a fundamental question is to determine what properties of the group are captured by the topology of the boundary at infinity. For example, the famous {\it Cannon conjecture} postulates that a hyperbolic group whose boundary at infinity is the $2$-sphere $S^2$ must support a properly discontinuous, isometric, cocompact action on hyperbolic $3$-space $\mathbb H^3$. 

In \cite{kk}, Kapovich and Kleiner study groups whose boundary at infinity is a Sierpinski carpet -- a boundary version of the Cannon conjecture. In \cite{blw}, Bartels, L\"uck, and Weinberger study groups whose boundary at infinity is a sphere $S^n$ of dimension
$n\geq 5$ -- a high-dimensional version of the Cannon conjecture. In this paper, we consider groups whose boundary at infinity
are high-dimensional Sierpinski spaces -- thus lying somewhere between the work of Kapovich-Kleiner and that of Bartels-L\"uck-Weinberger.

The two main theorems are as follows. Let $\scr{S}^{n-2}$ denote an $(n-2)$-dimensional \emph{Sierpinski space}. See Section \ref{sec:sierpinski} for details.

\begin{thm}\label{thm:main}
Fix $n\ge7$ and let $G$ be a torsion-free hyperbolic group. If the visual boundary $\pa_\infty G$ is homeomorphic to $\scr S^{n-2}$, then there exists an $n$-dimensional compact aspherical topological manifold $W$ with nonempty boundary such that $\pi_1(W)\cong G$. 
\end{thm}

Note that the fundamental group $\pi$ of a closed aspherical manifold $M$ is an example of a Poincar\'e duality group. Whether or not all finitely presented Poincar\'e duality groups arise in this fashion is an open problem that goes back to Wall \cite{wall_problems}. Theorem \ref{thm:main} addresses a relative version of this problem for a special class of groups.

Our second result shows that the converse of Theorem \ref{thm:main} is false:

\begin{thm}\label{thm:converse}
For each $n\geq 4$, there exists a compact aspherical manifold $M^n$ with nonempty boundary such that $G=\pi_1(M)$ is a hyperbolic group, but $\pa_\infty G$ is {\bf not} homeomorphic to $\scr S^{n-2}$. 
\end{thm}

\vskip 10 pt 

\noindent {\it Structure of paper.} In Section \ref{sec:sierpinski} we recall the definition of an $n$-dimensional Sierpinski space. In Sections \ref{sec:main} and \ref{sec:converse}, we prove Theorems \ref{thm:main} and \ref{thm:converse}, respectively. In Section \ref{sec:CAT0}, we remark on a generalization of Theorem \ref{thm:main} to CAT(0) groups. 

\vskip 10pt 

\noindent{\it Acknowledgments.} The authors wish to thank S.\ Weinberger and S.\ Ferry for useful conversations about surgery theory, and M. Davis for helpful discussions on approximating cellular maps. The authors thank B.\ Farb for comments on a draft of this paper. The first author was partially supported by the NSF, under grant DMS-1207782, and the second author was partially supported by the NSF grant DGE-1144082.

\section{$n$-dimensional Sierpinski space and hyperbolic groups}\label{sec:sierpinski}

We use Cannon's definition of $n$-dimensional Sierpinski space \cite{cannon_sierpinski} (Cannon uses the term Sierpinski \emph{curve} instead of Sierpinski \emph{space}). 

\begin{defn} Fix $n\ge0$. Let $D_1,D_2,\ld\sbs S^{n+1}$ be a sequence of open topological balls such that 
\begin{enumerate} 
\item[(i)] $\ov{D_i}\cap\ov{D_j}=\es$ for $i\neq j$, 
\item[(ii)] $\text{diam}(D_i)\ra0$ with respect to the round metric on $S^{n+1}$, and 
\item[(iii)] $\bigcup D_i\sbs S^{n+1}$ is dense. 
\end{enumerate}Then $\scr S^{n}:=S^{n+1}\bs\bigcup D_i$ is an \emph{$n$-dimensional Sierpinski space}. The spheres $S^{n}\cong\pa(\ov{D_i})\sbs\scr S$ are called \emph{peripheral spheres}. 
\end{defn}
\vspace{.1in}
\noindent{\it Example.} A 0-dimensional Sierpinski space $\scr S^0$ is a Cantor set, while the space $\scr S^1$ is the classical Sierpinski carpet. The Sierpinski space $\scr S^{n-2}$ arises as the visual boundary of hyperbolic groups (in the sense of Gromov \cite{gromov_hypgps}). For example, if $W^{n}$ is a hyperbolic $n$-manifold with nonempty totally geodesic boundary, then $\pi_1(W)$ is a hyperbolic group whose visual boundary is a Sierpinski $(n-2)$-space. To see this, observe that the universal cover $\wtil W$ can be embedded as a submanifold of hyperbolic space $\wtil W\hra \hy^n$. Using the disk model, the visual boundary $\pa_\infty\wtil W$ is a subspace of $\pa_\infty\hy^n\cong S^{n-1}$. The boundary components of $W$ lift to countably many disjoint geodesic hyperplanes $\hy^{n-1}\sbs\hy^n$. Each hyperplane has boundary a sphere $\pa_\infty\hy^{n-1}\cong S^{n-2}$, which bounds an open ball $\bb D^{n-1}\sbs S^{n-1}$. The visual boundary of $\wtil W$ is obtained by removing this countable collection of open balls, yielding a Sierpinski space $\scr S^{n-2}$. 

The simplest example of this is when $W$ is a torus with one boundary component (see Figure \ref{fig1}).
\begin{figure}[h!]
  \centering
    \includegraphics[width=0.35\textwidth]{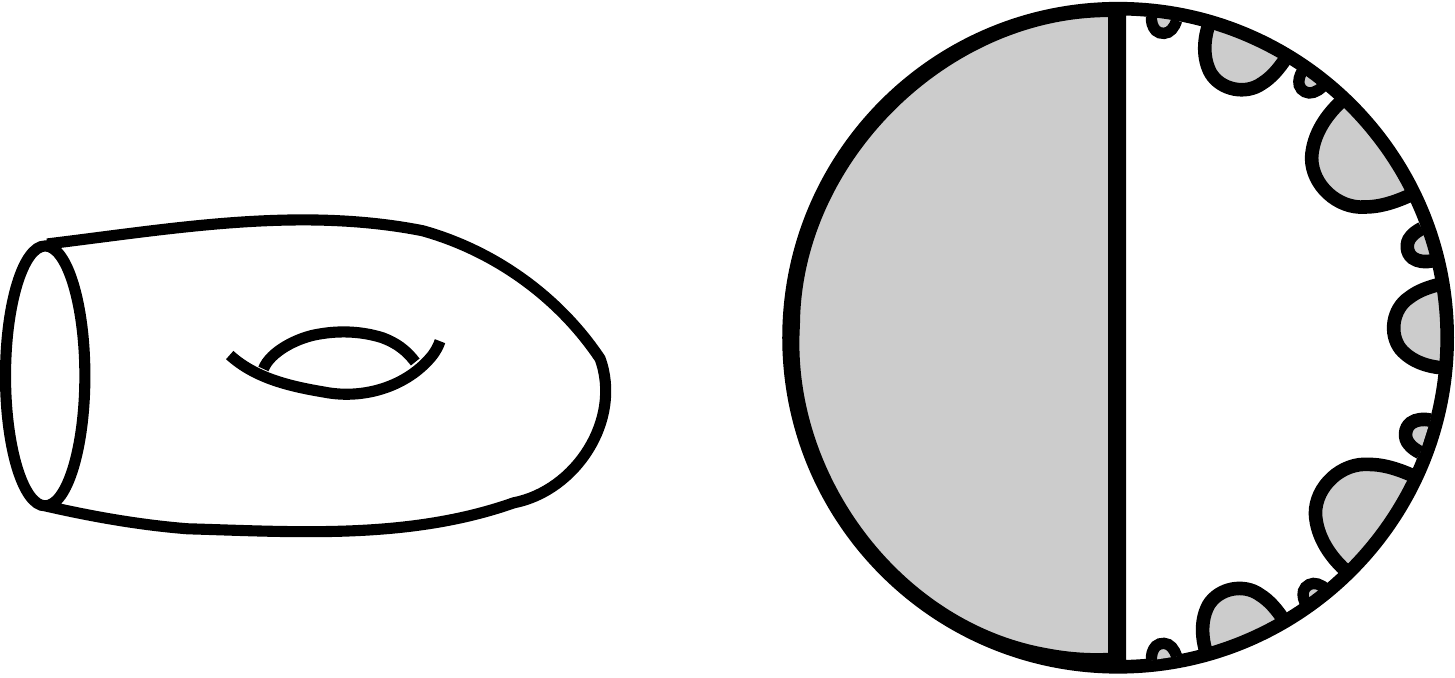}
  \caption{A torus with one boundary component, and its universal cover inside the hyperbolic plane.}\label{fig1}
\end{figure}
More examples are furnished by the following general theorem of Lafont \cite{lafont_boundaryCartanHadamard}.

\begin{thm}[Lafont]\label{thm:lafont}
Let $M^n$ be a compact, negatively curved Riemannian manifold with nonempty totally geodesic boundary. Then $\pa_\infty\wtil M$ is homeomorphic to $\scr S^{n-2}$.
\end{thm}
We remark that the dimension restriction in the statement of \cite[Theorem 1.1]{lafont_boundaryCartanHadamard} is unnecessary thanks to work of Freedman and Quinn (c.f.\ the MathSciNet review of \cite{ruane}). \\

\section{Proof of Theorem \ref{thm:main}}\label{sec:main}

\begin{proof} We proceed in three steps. 

\vskip 5pt

\noindent{\bf Step 1 (Peripheral subgroups and Poincar\'e duality pairs).} Recall that $G$ is a torsion-free hyperbolic group such that $\pa_\infty G\cong\scr S^{n-2}$. The stabilizer $H\le G$ of a peripheral sphere $S^{n-2}\sbs\scr S^{n-2}$ is called a \emph{peripheral subgroup}. By Kapovich-Kleiner \cite[Theorem 8(1)]{kk}, there are finitely many peripheral subgroups, up to conjugacy in $G$. Choose representatives $H_1,\ld,H_p$ for the conjugacy classes. 

In order to show that $G$ is the fundamental group of a manifold with boundary, we first need to establish that $G$ has the same Poincar\'e duality as a manifold with boundary. To be precise, Kapovich-Kleiner \cite[Corollary 12]{kk} show that $(G,\{H_i\})$ is a \emph{group PD($n$) pair} in the sense of Bieri-Eckmann \cite{bieri+eckmann_pdpairs}. This has the following topological consequence (see \cite[Theorem 1]{johnson+wall} and \cite[Section 6]{bieri+eckmann_duality}): let $(X,Y)$ be the CW-complex pair obtained by taking $Y=\coprod_{i=1}^p BH_i$ and defining $X$ to be the mapping cylinder of the map $\coprod BH_i\ra BG$. Then $(X,Y)$ is a \emph{CW-complex PD($n$) pair} in the sense of Wall \cite{wall_pc}. In particular this means that there are isomorphisms  $H^i(X;\z)\cong H_{n-i}(X,Y;\z)$ and $H^{i-1}(Y;\z)\cong H_{n-i}(Y;\z)$ induced by cap product with $[X]\in H_n(X)$ and $\pa[X]\in H_{n-1}(Y)$, respectively, and that $X$ is a \emph{finitely dominated} CW complex (i.e.\ there exists a finite CW complex $L$ and maps $X\xra i L\xra r X$ such that $r\circ i=\text{id}_X$). 

\vskip 10pt

\noindent{\bf Step 2 (Preparing for surgery).} Let $(X,Y)$ be the pair from Step 0. We now explain why $(X,Y)$ is homotopy equivalent to a pair $(K,N)$ such that \begin{enumerate}[(A)]\item $K$ is a \emph{finite} CW complex, and \item $N$ is a manifold.\end{enumerate} This will allow us to employ the total surgery obstruction in Step 3. 

\vskip 5 pt

(A) Wall's finiteness obstruction $\til o(X)\in\wtil K_0(X)$ vanishes if and only if $X$ is homotopy equivalent to a finite CW complex \cite{wall_finitenessobstruction}. Thus to show (A), it suffices to show $\wtil K_0(X)=0$. This is a corollary of the following powerful result (see \cite[Proof of Theorem 1.2]{blw} for more information): 

\begin{thm}[Bartels-L\"uck \cite{bl_bc}, Bartels-L\"uck-Reich \cite{blr}]\label{thm:blr}
Let $G$ be a torsion-free hyperbolic group $G$. Then 
\begin{itemize}
\item[($\dagger$)] the (non-connective) $K$-theory assembly map $H_i(BG;\bb K_\z)\ra K_i(\z G)$ is an isomorphism for $i\le 0$ and surjective for $i=1$; 
\item[($\ddagger$)] the (non-connective) $L$-theory assembly map $H_i\big(BG;^w\bb L_\z^{\lan-\infty\ran}\big)\ra L_i^{\lan-\infty\ran}(\z G,w)$ is bijective for every $i\in\z$ and every orientation homomorphism $w:G\ra\{\pm1\}$. 
\end{itemize}
\end{thm}

The conditions ($\dagger$) and ($\ddagger$) are called the Farrell-Jones conjectures in $K$- and $L$-theory, respectively.

(B) It remains to see that $Y$ is homotopy equivalent to a closed manifold $N^{n-1}$. By definition $Y$ is homotopy equivalent to $\coprod_{i=1}^p BH_i$. The peripheral subgroups $H_i$ are all hyperbolic groups, and $\pa_\infty H_i$ is identified with the sphere $S^{n-2}\sbs\scr S^{n-2}$ stabilized by $H_i$ (see \cite[Theorem 8]{kk}). The following result from \cite[Theorem A]{blw} implies that $Y\simeq\coprod_{i=1}^p BH_i$ is homotopy equivalent to a manifold:

\begin{thm}[Bartels-L\"uck-Weinberger \cite{blw}]\label{thm:blw}Fix $n\ge7$, and let $H$ be a torsion-free hyperbolic group. If $\pa_\infty H\cong S^{n-2}$, then there is a closed aspherical manifold $N^{n-1}$ such that $\pi_1(N)\cong H$. 
\end{thm}

\noindent{\bf Step 3 (The total surgery obstruction).} Let $(K,N)$ be the pair from Step 2. The \emph{structure set} $S^{TOP}(K)$ is defined as the set of equivalence classes of homotopy equivalences $f:(M,\pa M)\ra(K,N)$ where $(M,\pa M)$ is a manifold with boundary and $f\rest{}{\pa M}:\pa M\ra N$ is a homeomorphism (the equivalence relation is $h$-cobordism rel $\pa$; see \cite[Chapter 18]{ranicki_algebraicLtheory}). Surgery theory provides computable obstructions to determine whether or not $(K,N)$ is homotopy equivalent to a manifold with boundary, i.e.\ whether or not $S^{TOP}(K)\neq\es$. 

We will follow the algebraic approach detailed in Ranicki \cite{ranicki_algebraicLtheory}. The \emph{total surgery obstruction} $s_\pa(K)$ lives in the \emph{structure group} $\bb S_n(K)$ and has the property that $s_\pa(K)=0$ if and only if $(K,N)$ is homotopy equivalent (rel boundary) to an $n$-manifold with boundary; see \cite[Theorem 1]{ranicki_tso}. The group $\bb S_n(K)$ fits into the \emph{algebraic surgery exact sequence} \cite[Definition 15.19]{ranicki_algebraicLtheory}
\[\cd\ra H_n(K;\bb L_\bullet)\xra A L_n\big(\pi_1(K)\big)\ra \bb S_n(K)\ra H_{n-1}(K;\bb L_\bullet)\ra\cd\]
where $A$ is the \emph{assembly map} and $\bb L_\bullet$ is the \emph{1-connective surgery spectrum} whose 0th space is $G/TOP$ and whose homotopy groups are $\pi_i(\bb L_\bullet)=L_i(\z)$ for $i\ge1$. 

To show that $S^{TOP}(K)\neq\es$, we will show that $\bb S_n(K)=0$. For this, we need to consider two other versions of the structure groups. 

\begin{itemize} 
\item 
The \emph{quadratic structure groups} $\bb S_i(\z,K)$ are defined in \cite[Definition 14.6]{ranicki_algebraicLtheory}. 

\item The group $\ov{\bb S}_n(K)$ (see \cite[Chapter 25]{ranicki_algebraicLtheory}) belongs to the \emph{4-periodic algebraic surgery exact sequence}
\[\cd\ra H_n(K;\ov{\bb L}_\bullet)\xra A L_n\big(\pi_1(K)\big)\ra \ov{\bb S}_n(K)\ra H_{n-1}(K;\ov{\bb L}_\bullet)\ra\cd\]
where $\ov{\bb L}_\bullet$ is the \emph{0-connective surgery spectrum} whose 0th space is $L_0(\z)\ti G/TOP\cong\z\ti G/TOP$ and whose homotopy groups are $\pi_i(\ov{\bb L}_\bullet)=L_i(\z)$ for $i\ge0$.

\end{itemize} 

\noindent In order to show that $\bb S_n(K)=0$, we need the following three facts. 
\begin{enumerate}[(a)] 
\item The groups $\ov{\bb S}_n(K)$ and $\bb S_n(\z,K)$ are equal. This follows directly from Ranicki \cite[Proposition 15.11(iii)-(iv)]{ranicki_algebraicLtheory}. Here we have used that $\dim K\ge6$. Note that $L_q(\z)=0$ for $q=-1$, and in Ranicki's notation $\bb S_n\lan 0\ran(\z,K)=\ov{\bb S}_n(K)$ (compare with \cite[Page 289]{ranicki_algebraicLtheory}). 
\item The quadratic structure groups $\bb S_i(\z, K)\cong \bb S_i(\z, BG)$ are 0 for all $i\in\z$. For the proof, see \cite[Proof of Theorem 1.2]{blw}. Note that this also uses Theorem \ref{thm:blr}. 
\item There is an exact sequence
\[H_n\big(K;L_0(\z)\big)\ra\bb S_n(K)\ra\ov{\bb S}_n(K).\]
See Ranicki \cite[Theorem 25.3(i)]{ranicki_algebraicLtheory}.
\end{enumerate}

From (a) and (b), it follows that $\ov{\bb S}_n(K)=0$. Then, by (c), to show $\bb S_n(K)=0$ is suffices to show $H_n(K;L_0(\z))= H_n(K;\z)=0$. This can be seen from the long exact sequence in homology of a pair $(K,N)$:
\[H_n(N;\z)\ra H_n(K;\z)\ra H_n(K,N;\z)\xra\pa H_{n-1}(N;\z).\]
The group $H_n(N;\z)=0$ because $N$ is a PD$(n-1)$ complex. Also $H_n(K,N;\z)\cong\z$ is generated by the fundamental class $[K]$, and $\pa[K]$ is a sum of fundamental classes of the components of $N$. In particular $\pa[K]\neq0$, so $H_n(K;\z)=0$, as desired. 

This concludes the proof of Theorem \ref{thm:main}.
\end{proof}

\section{Proof of Theorem \ref{thm:converse}}\label{sec:converse} 

The proof of Theorem \ref{thm:converse} is an adaptation of \cite[Section (5a), (5c)]{dj}. We briefly explain the relative version of \cite{dj} and the problem with extending it directly to our case. 

The paper \cite{dj} uses hyperbolization to construct a closed, locally CAT(-1) manifold $M^n$ with the unusual property that $\pa_\infty \wtil M$ is {\bf not} homomorphic to $S^{n-1}$. To show this, they establish that $\pa_\infty\wtil M-\{\ga_+,\ga_-\}$ is not simply connected, where $\ga_+,\ga_-$ are the endpoints of a geodesic $\ga:(-\infty,\infty)\ra\wtil M$ whose link is a homology sphere $H$ with $\pi_1(H)\neq1$. In order to find nontrivial elements of $\pi_1\big(\pa_\infty\wtil M-\{\ga_+,\ga_-\}\big)$, \cite{dj} studies metric spheres $S_p(r)$ centered at $p=\ga(0)$. When $s>r$, there are natural \emph{geodesic contraction maps} $\rho_r^s:S_p(s)\ra S_p(r)$, which allow one to relate the topology of small spheres to the topology of $\pa_\infty\wtil M=\varprojlim \{S_p(r)\}_{r>0}$. The central property of the maps $\rho_{r}^s$ that makes the comparison work is that they are \emph{cell-like}. 

Following \cite{dj}, we will construct a triangulated, locally CAT(-1) manifold $M$ with totally geodesic boundary $\pa M$ whose universal cover $\wtil M$ contains a geodesic $\ga:(-\infty,\infty)\ra\wtil M$ whose link is a homology sphere $H$ with $\pi_1(H)\neq1$. As above, we wish to show $\pi_1(\pa_\infty\wtil M-\{\ga_+,\ga_-\})\neq1$ (Lemma \ref{fund-grp-Sierpinski} below then implies that $\pa_\infty\wtil M$ is {\bf not} homeomorphic to $\scr S^{n-2}$). In this case $\wtil M$ is a manifold with boundary, and the maps $\rho_r^s:S_p(s)\ra S_p(r)$ are not surjective for $s>>r$. This prevents us from proceeding directly as in \cite{dj}. To bypass this issue, we ``cap off" the boundary components of $\wtil M$ to obtain a CAT(-1) manifold $\what M\supset\wtil M$ to which the arguments of \cite{dj} apply; in particular, $\pi_1(\pa_\infty\what M- \{\ga_+,\ga_-\})\neq1$. At this point it will be clear from the capping procedure (see specifically Lemma \ref{lem:assertion} below) that $\pi_1(\pa_\infty\wtil M- \{\ga_+,\ga_-\})\neq1$. 

For the proof of Theorem \ref{thm:converse}, we need the following elementary fact. 

\begin{lem}\label{fund-grp-Sierpinski}
For $n\geq 2$, the $n$-dimensional Sierpinski space $\scr S^{n}$ is simply-connected. Moreover, if $F \subset \scr S^n$ is any finite collection of points in $\scr S^n$, then $\scr S^n \setminus F$ is still simply-connected.
\end{lem}

\begin{proof}
Model $\scr S^n$ as the complement, in the standard sphere $S^{n+1}$, of the interiors of a dense collection of pairwise disjoint round disks $D_i$ whose radii $r_i$ tend to zero. If $\gamma$ is a curve in $\scr S^n \subset S^{n+1}$, we can find a bounding disk $\phi: \mathbb D^2 \rightarrow S^{n+1}$. Perturbing the map a little bit, we can assume that $\phi$ is transverse to all the $D_i$. Inductively define $\phi_k: \mathbb D^2\rightarrow S^{n+1}$ to have image disjoint from $D_1, \ldots , D_k$, as follows. $\phi^{-1}(\partial D_k)$ is a finite collection of curves in $\mathbb D^2$, and each of these curves maps to a curve $\eta_j$ on $\partial D_k\cong S^n$. Since $n\geq 2$, we can redefine $\phi_{k-1}$ on the interior of these finitely many curves in $\mathbb D^2$, by sending each of these to a bounding disk in $\partial D_k$ for the corresponding $\eta_j$. Since the diameter of the $D_i$ shrinks to zero, the maps $\phi_k$ converge to a map $\phi_\infty:\mathbb D^2 \rightarrow S^{n+1}$ whose boundary coincides with $\gamma$, and whose image is disjoint from
the interiors of all the $D_i$, i.e. the image of $\phi_\infty$ lies in $\scr S^n$. A similar argument works even after removing finitely many points in $\scr S^n$.
\end{proof}

\begin{proof}[Proof of Theorem \ref{thm:converse}]\mbox{ } We proceed in several steps.

\vskip 5pt 

\noindent{\bf Step 1 (Construction).} We construct $M$ using the \emph{strict hyperbolization} construction of Charney-Davis \cite{cd}. For simplicity we will focus primarily on the case $n\ge5$. The case $n=4$ will be explained at the end of Step 2. 

The case $n\geq 5$ is modeled on \cite[Section (5c)]{dj}. Fix a smooth $n$-manifold $X$ with non-empty connected boundary $Y$, equipped with a PL-triangulation. Choose a smooth homology sphere $H^{n-2}$ with non-trivial fundamental group, take a PL-triangulation of $H$, and consider the double suspension $\Sigma^2 H \cong S^n$, with the obvious induced (no longer PL) triangulation. Take the triangulated connect sum $X \sharp \Sigma^2 H$, obtained by using the interior of a pair of $n$-simplices in the triangulated $X$, $\Sigma^2 H$ to take the connect sum (and chosen so that simplex in $X$ does not intersect the boundary of $X$). Note that, topologically $X \sharp \Sigma^2 H$ is homeomorphic to $X$, but now has a triangulation that fails to be PL -- there is precisely one $4$-cycle in the $1$-skeleton of the triangulation whose link is $H$ (instead of $S^{n-2}$). Finally, we let $M^n = h(X \sharp \Sigma^2 H)$, an $n$-manifold with boundary $N^{n-1}=h(Y)$, and set $G= \pi_1(M)$. 

\vskip 10 pt

\noindent{\bf Step 2 (Capping procedure).} To show that $\pa_\infty G\neq\scr S^{n-2}$, first identify $\pa_\infty G\cong\pa_\infty\wtil M$. We use Lemma \ref{fund-grp-Sierpinski} and show that $\pi_1(\pa_\infty\wtil M\setminus F)\neq1$, where $F=\{\ga_+,\ga_-\}$ consists of two points. 

$\wtil M$ is a non-compact CAT(-1) manifold with non-empty boundary, each component of which is isometric to $\widetilde {h(Y)}$. To understand $\pa_\infty\wtil M$, we first define an isometric embedding $\wtil M\hra\what M$ into a CAT(-1) space without boundary. It will be easier to analyze $\what M$, which is obtained from $\wtil M$ by gluing a certain space $Z$ to each component of $\pa_\infty\wtil M$. Next we define $Z$ and describe its key features. 

Let $DX$ be the double of $X$ across $Y$, with the induced triangulation. We apply a strict hyperbolization of Charney-Davis \cite{cd} to obtain a closed $n$-manifold $h(DX)$ equipped with a locally CAT(-1) metric. The universal cover $\widetilde {h(DX)}$ has boundary at infinity homeomorphic to $S^{n-1}$ (see \cite[Theorem (3b.2)]{dj}). Take any lift $\widetilde {h(Y)}$ of the separating codimension one submanifold $h(Y) \subset h(DX)$. Then $\widetilde {h(Y)}$ separates $\widetilde {h(DX)}$ into two (isometric) convex subsets. Denote by $Z$ the closure of one of these convex subsets. Then $Z$ is a non-compact locally CAT(-1) $n$-manifold with totally geodesic boundary $\widetilde {h(Y)}$.  

\begin{lem}\label{lem:assertion}
The boundary at infinity $\partial _\infty Z$ of $Z$ is homeomorphic to $\mathbb D^{n-1}$. The 
inclusion $\widetilde {h(Y)} = \partial Z$ induces, at the boundary at infinity, an identification $\partial_\infty \widetilde {h(Y)}= S^{n-2} = \partial (\mathbb D^{n-1})$.
\end{lem}




Let us momentarily assume Lemma \ref{lem:assertion} and finish the proof. Form the CAT(-1) space $\widehat M$ by gluing a copy of $Z$ to each boundary component of $\partial \wtil M$, by isometrically identifying the copy of $\wtil {h(Y)}$ inside $Z$ with the boundary component. We have an isometric embedding $\wtil M \hookrightarrow \what M$, inducing an embedding $\partial _\infty \wtil M \hookrightarrow \partial _\infty \what M$. Let $\gamma$ be a lift, in $\wtil M \subset \what M$ of the singular geodesic in $M$, i.e. the geodesic whose link is the homology sphere $H$. The argument in \cite[Proof of Theorem 5c.1(iv), pg.\ 385]{dj} applies verbatim to show that $\partial _\infty \what M -\{\gamma_+, \gamma_-\}$ is not simply-connected. If $\eta$ denotes a homotopically non-trivial loop in $\partial _\infty \what M - \{\gamma_+, \gamma_-\}$, then Lemma \ref{lem:assertion} allows us to use the same argument as in Lemma \ref{fund-grp-Sierpinski} to homotope $\eta$ into the subset $\partial _\infty \wtil M= \partial_\infty G$. We conclude that $\partial _\infty G -\{\gamma_+, \gamma_-\}$ fails to be simply connected. From Lemma \ref{fund-grp-Sierpinski}, we conclude that $\partial _\infty G$ is not homeomorphic to $\scr S^{n-2}$.

The $n=4$ case proceeds similarly, but is modeled instead on \cite[Section (5a)]{dj}. Briefly, one lets $X$ be a $4$-dimensional simplicial complex whose geometric realization is a homology manifold with non-empty boundary $Y$, and which contains a singular point in the interior of $X$ (whose link is, for example, the Poincar\'e homology $3$-sphere $H$). One then looks at the universal cover of the hyperbolization $W= h(X)$. We can ``cap off'' the boundary components of $\wtil W$ as in the last paragraph to obtain $\widehat W$. Then the arguments in \cite[Section 3d]{dj} shows that $\pi_1(\partial _\infty \widehat W)$ is non-trivial. Again, using Lemma \ref{lem:assertion}, we can push a homotopically non-trivial loop in $\partial _\infty \widehat W$ into the subset $\partial _\infty \wtil W = \partial _\infty G$. From  Lemma \ref{fund-grp-Sierpinski},
we conclude that $\partial _\infty G$ is not homeomorphic to $\scr S^{2}$. Finally, even though $W$ is not a manifold, it is homotopy equivalent to a manifold: just remove a small neighborhood of the singular cone point, and replace it by a contractible manifold which bounds $H$. The resulting $4$-manifold $M$ has the desired properties.

\vskip 10pt

\noindent{\bf Step 3 (Reducing Lemma \ref{lem:assertion}).} To complete the proof of the theorem, we are left with verifying Lemma \ref{lem:assertion}. This is again a minor adaptation of the arguments in \cite[Sections 3b, 3c]{dj}. Choose a basepoint $x\in \partial Z$, and consider the closed metric $r$-balls $\ov{B}_Z(r)$, $\ov B_{\partial Z}(r)$ in the spaces $Z$, $\partial Z$, centered at $x$, as well as the metric $r$-spheres $S_Z(r)$ and $S_{\partial Z}(r)$. The proof of Lemma \ref{lem:assertion} will rely on the following:

\vskip 5pt

\noindent \underline{Claim 1:} For all $r$, the metric spheres $S_Z(r)$ are manifolds with boundary $S_{\partial Z}(r)$.

\vskip 5pt

\noindent \underline{Claim 2:} For points $p\in S_{\partial Z}(r)$, the complement $\Lk(p) \setminus B_{\Lk(p)}(v; \pi)$ of the
metric ball of radius $\pi$, centered at $v \in \partial \left(\Lk(p)\right)$ in the link of $p$, is a cell-like set.

\vskip 5pt

From these two Claims, it is easy to conclude. If one takes a small enough $r$, then clearly $S_Z(r)$ is homeomorphic to a disk $\mathbb D^{n-1}$.  In view of Claim 2 and the discussion in \cite[pg.\ 372]{dj}, there is an $\epsilon >0$ such that each of the geodesic contraction maps $\rho_r^s: S_Z(s) \rightarrow S_Z(r)$ is a cell-like map when $r< s< r+\epsilon$. So by Claim 1, the maps $\rho_r^s$ are cell-like maps between manifolds with boundaries. From the work of Siebenmann \cite{sieb}, Quinn \cite{quinn}, and Armentrout \cite{arm} it follows that each $\rho_r^s$ is a \emph{near-homeomorphism} (i.e. can be approximated arbitrarily closely by homeomorphisms), and hence, that all the $S_Z(r)$ are homeomorphic to a disk $\mathbb D^{n-1}$, with boundary $\partial S_Z(r) = S_{\partial Z}(r)$.

Finally, we identify the pair $\left(\partial_\infty Z , \partial _\infty (\partial Z)\right)$ with the inverse limit $\varprojlim \left\{\left(S_Z(r), S_{\partial Z}(r) \right) \right\}_{r>0}$, where the bonding maps are given by the maps $\rho_r^s$ (where $0<r<s$), which we saw are all near-homeomorphisms. Lemma \ref{lem:assertion} now follows by applying the main result of Brown \cite{brown1}. 

This reduces the proof of Lemma \ref{lem:assertion} (and hence also of the theorem) to checking Claim 1 and Claim 2 -- which are the last two steps of the proof.

\vskip 5pt

\noindent {\bf Step 4 (Proof of Claim 1).} We first argue that the ball $B_Z(r)$ of radius $r$ is a manifold with boundary. It is clear that points $p \in \text{Int}(\wtil M)$ at distance $<r$ from the basepoint have manifold neighborhoods. It is also immediate that points $p\in \partial \wtil M$ at distance $<r$ from the basepoint have neighborhoods homeomorphic to $\mathbb R^{n-1}\times \mathbb R_+$. Points at distance $=r$ from the basepoint are either in $\text{Int}(\wtil M)$ or on $\partial \wtil M$. 
For points $p$ in $\text{Int}(\wtil M)$, the argument in \cite[pg.\ 372]{dj} shows that $p$ has a neighborhood homeomorphic to $\mathbb R^{n-1} \times \mathbb R_+$. So the only possible points to worry about are points at distance $=r$, and lying on the subset $\partial \wtil M$. But for such a point $p$, a similar argument works with no trouble. Let $v$ be the point in $\Lk(p)$ pointing from $p$ to the basepoint $x$, and consider the closed ball $\overline{B}_{\Lk(p)}(v; \pi/2)$ in the link of $p$, centered at $v$, of radius $\pi/2$. For any vector $w\in \overline{B}_{\Lk(p)}(v; \pi/2)$, one can look at the geodesic $\gamma_w$ emanating from $p$, in the direction $w$ ($\gamma_w$ is well-defined close to $p$). If the direction $w$ is at distance $<\pi/2$ from $v$, then for a small interval 
of time $[0,s(w)]$, the geodesic $\gamma_w$ lies entirely in $B_Z(r)$, with $\gamma_w\left(s(w) \right)\in S_Z(r)\cup B_{\pa Z}(r)$. Note that $s$ varies continuously and $s(w)\to 0$ as $w\to S_{\Lk(p)}(v; \pi/2)$.
It follows that $p$ has a neighborhood homeomorphic to the set $\hat X$ constructed as follows: take the product $I\times \overline{B}_{\Lk(p)}(v; \pi/2)$, collapse the fibers over the subset $S_{\Lk(p)}(v; \pi/2)$ to $0$, and then collapse the subset $\{0\} \times  \overline{B}_{\Lk(p)}(v; \pi/2)$ to a single point (which is identified with $p$) -- see Figure \ref{fig:claim1}. By an inductive argument (note that $\dim (\Lk(p))=\dim (\wtil M)-1$) one can assume that $\overline{B}_{\Lk(p)}(v; \pi/2)$ is homeomorphic to a disk $\mathbb D^{n-1}$, with the subset $S_{\Lk(p)}(v; \pi/2)$ corresponding to an embedded $\mathbb D^{n-2}$ inside $\partial \mathbb D^{n-1} \cong S^{n-2}$. Following the construction of $\hat X$ given above, we see that $\hat X$ is homeomorphic to $\mathbb D^{n}$, with the point corresponding to $p$ lying on $\partial \mathbb D^n$. This shows that $B_Z(r)$ is indeed a manifold with boundary, and that the boundary of $B_Z(r)$ naturally decomposes as the union of $S_Z(r) \cup B_{\partial Z}(r)$, where the union is over the common subset $S_{\partial Z}(r)$. 

Finally, we check that $S_Z(r)$ is an $(n-1)$-manifold with boundary. For points $p\in S_Z(r)$ lying in $\text{Int}(\wtil M)$, it follows easily from \cite[pg.\ 372]{dj} that these points have neighborhoods homeomorphic to $\mathbb D^{n-1}$ with $p$ lying as an interior point. In the case where $p \in S_Z(r)$ lies on $\partial \wtil M$, we look at the neighborhood $\hat X$ of $p$ constructed above. Within $\hat X$, the subset corresponding to $S_Z(r)$ consists of (the image of) a small neighborhood $U$ of $\{1\} \times S_{\Lk(p)}(v; \pi/2) \cong \mathbb D^{n-2}$ inside the slice $\{1\} \times \overline{B}_{\Lk(p)}(v; \pi/2)\cong \mathbb D^{n-1}$. Note that the $(n-2)$-disk $S_{\Lk(p)}(v; \pi/2)$ lies in the boundary sphere of the $(n-1)$-disk $ \overline{B}_{\Lk(p)}(v; \pi/2)$ (by induction). The image of $U$ thus gives a copy of $\mathbb D^{n-1}$, with $p$ lying in the boundary of $\mathbb D^{n-1}$. Moreover, the subset of $U$ corresponding to $S_{\partial Z}(r)$ is just a neighborhood of $p$ inside the boundary sphere of $\mathbb D^{n-1}$, i.e. is homeomorphic to $\mathbb D^{n-2}$. This completes the argument for Claim 1.

\begin{figure}[h!] 
\labellist 
\small\hair 2pt 
\pinlabel $v$ at -15 485
\pinlabel $B_{L}(v;\pi/2)$ at -85 550 
\pinlabel $S_{L}(v;\pi/2)$ at 400 640
\pinlabel $\{0\}\ti\bb D^{n-1}$ at 585 460
\pinlabel $\{1\}\ti\bb D^{n-1}$ at 585 700
\pinlabel $U$ at 1040 700
\endlabellist 
\includegraphics[scale=0.3]{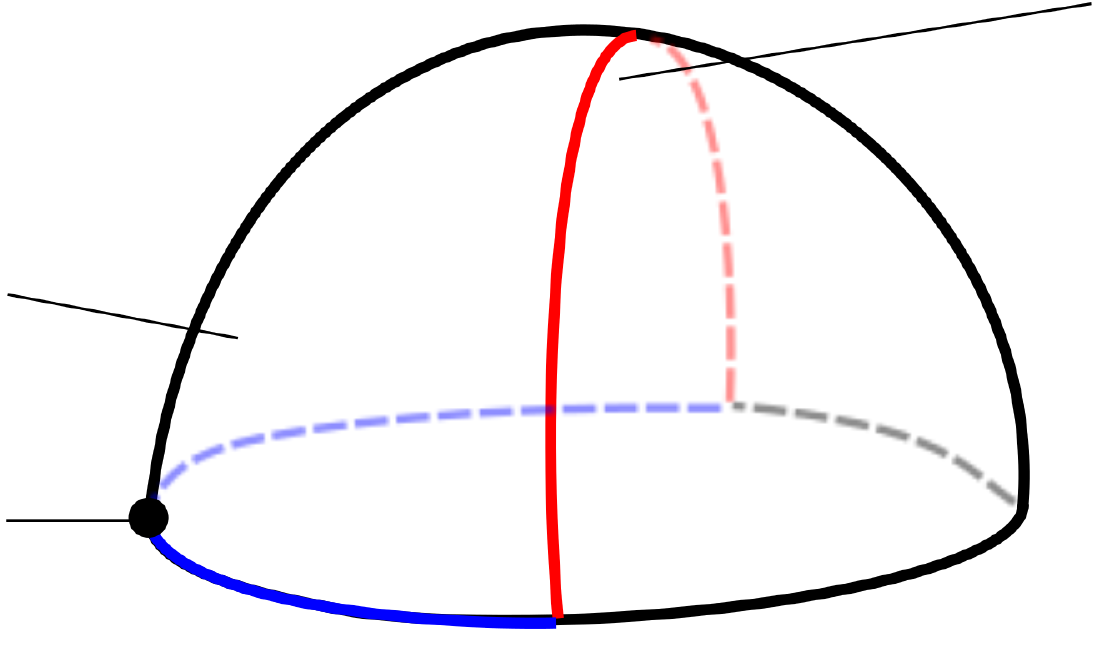} \hspace{1.35in}
\includegraphics[scale=0.3]{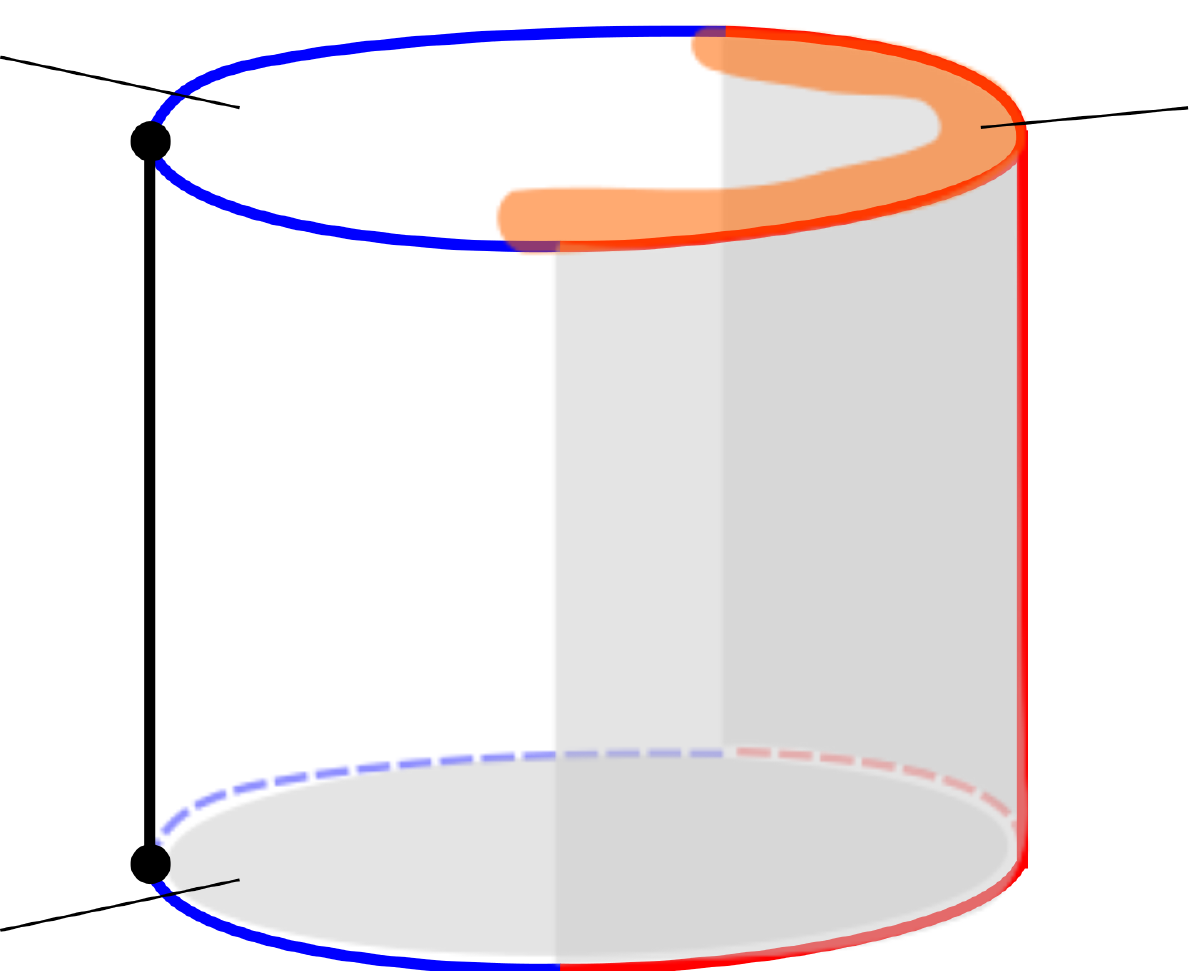} 
\caption{Left: The link $L=\Lk(p)$. Right: The space $I\ti\ov{B}_{\Lk(p)}(v;\pi/2)$, which is identified with a neighborhood $\hat X$ of $p$ after quotienting by the gray region.}  
\label{fig:claim1}
\end{figure}
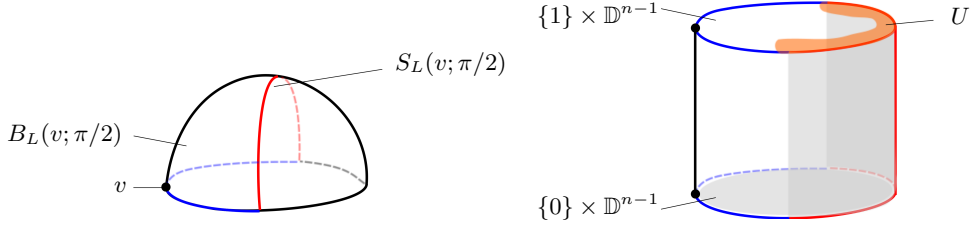 

\vskip 10pt

\noindent {\bf Step 5 (Proof of Claim 2).} We want to show that the complement $\Lk(p) \setminus B_{\Lk(p)}(v; \pi)$  is cell-like. The set $\Lk(p)$ is homeomorphic to a disk $\mathbb D^{n-1}$, so we can think of the set we are interested in as lying within the double $D\left(\Lk(p)\right) \cong S^{n-1}$. Given an $r\in (0, \pi)$, consider the subset $U_r\subset D\left(\Lk(p)\right) \cong S^{n-1}$ defined to be the union of $D\left(\Lk(p)\right) \setminus \Lk(p)$ and the set $B_{\Lk(p)}(v; r)$. See Figure \ref{fig:link}. We will show each such $U_r$ is homeomorphic to $\mathbb R^{n-1}$. Then by a result of Brown \cite{brown2} it follows that the union $U_\infty :=\bigcup_{r\in (0, \pi)}U_r \subset D\left(\Lk(p)\right) \cong S^{n-1}$ is also homeomorphic to $\mathbb R^{n-1}$. But if a subset of $S^{n-1}$ is homeomorphic to $\mathbb R^{n-1}$, its complement is automatically cell-like. Since the complement of $U_\infty$ coincides with $\Lk(p) \setminus B_{\Lk(p)}(v; \pi)$, this would establish Claim 2. 

To see that each $U_r$ is homeomorphic to $\mathbb R^{n-1}$, we consider their closures $\ov U_r$. We have that $U_r=\text{Int}(\ov U_r)$, and that $\ov U_r$ can be written as the union of a copy of $\Lk(p)$ along with $\overline{B}_{\Lk(p)}(v; r)$, where the union is taken over the common subset $\overline{B}_{\partial \Lk(p)}(v; r)$. Let us analyze the two pieces in this decomposition.

On one of the sides, the subset $\Lk(p)$ is homeomorphic to $\mathbb D^{n-1}$, and the common subset $\overline{B}_{\partial \Lk(p)}(v; r)$ is homeomorphic to an embedded $(n-2)$-disk $\mathbb D^{n-2}$ inside the boundary sphere $\partial \Lk(p) \cong S^{n-2}$. Note that, by varying the parameter $r$, we see that 
\[S^{n-3}\simeq\pa\ov B_{\pa\Lk(p)}(v;r)\sbs\pa\Lk(p)\simeq S^{n-2}\]
is bicollared. On the other side, the subset $\overline{B}_{\Lk(p)}(v; r)$ is also homeomorphic to $\mathbb D^{n-1}$, and the gluing disk $\mathbb D^{n-2} \cong \overline{B}_{\partial \Lk(p)}(v; r)$ inside the boundary sphere $S^{n-2}\cong \partial \overline{B}_{\Lk(p)}(v; r)$ also has complement a disk (by the argument in Claim 1). Thus, we see that $\ov U_r$ is obtained by gluing together two closed $(n-1)$-disks, by identifying together two copies of an $(n-2)$-disk, where each copy is nicely embedded in the respective boundary spheres $S^{n-2}\cong \mathbb D^{n-1}$. It follows that $\ov U_r$ is also homeomorphic to $\mathbb D^{n-1}$. This completes the proof of Claim 2 and the proof of the theorem.
\end{proof} 

\begin{figure}[h!] 
\labellist 
\small\hair 2pt 
\pinlabel $v$ at -5 50
\pinlabel $B_L(v;r)$ at -40 115 
\pinlabel $L\setminus B_L(v;\pi)$  at 230 90 
\pinlabel $\ov{B}_{\pa L}(v;r)$ at 295 40
\pinlabel $\ov{B}_L(v;r)$ at 295 140
\pinlabel $\pa\>\ov{B}_{\pa L}(v;r)$ at 540 40
\pinlabel $DL\bs L$ at 515 8
\endlabellist 
\includegraphics[scale=0.5]{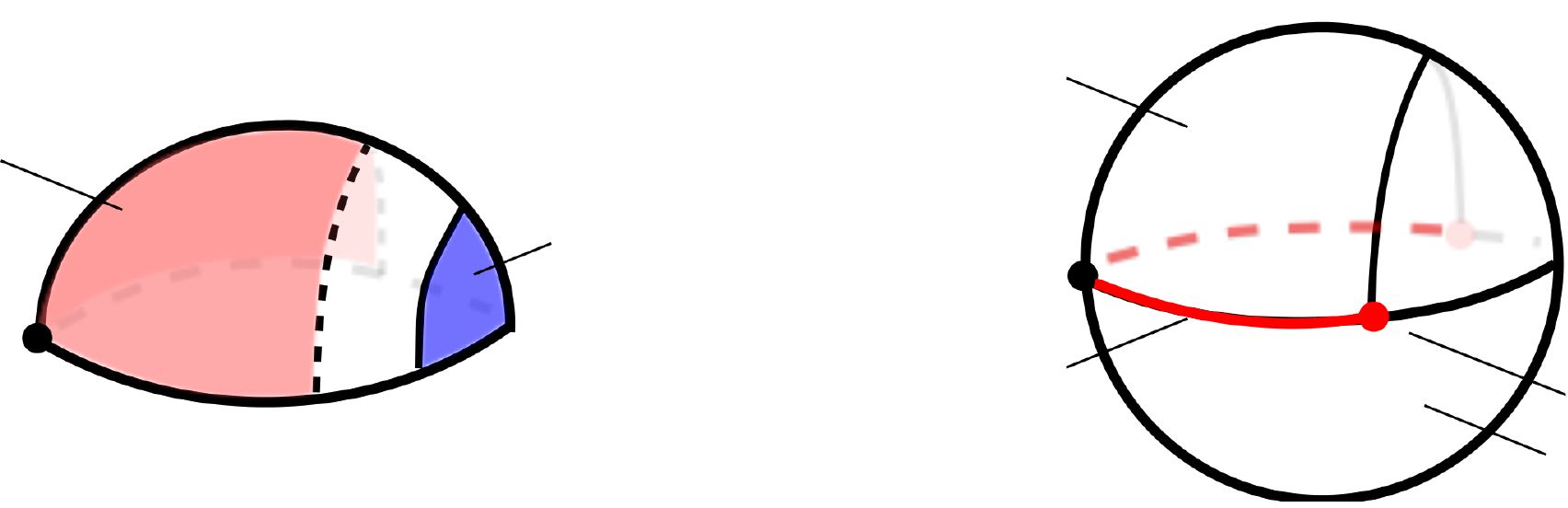} 
\caption{The link $L=\Lk(p)$ and its double $DL$. }  
\label{fig:link}
\end{figure}
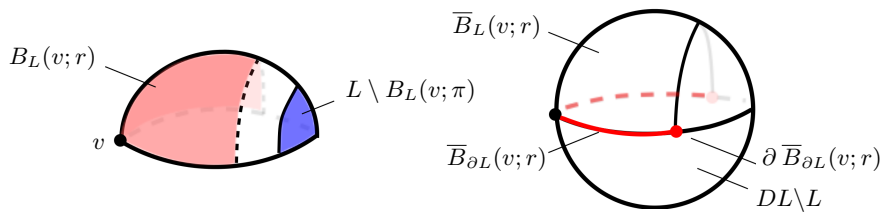

\begin{rmk}
Let us make a small comment on approximating cell-like maps by homeomorphisms, in the case of manifolds with boundary. The attentive reader will probably notice that, in Siebenmann's work \cite{sieb}, there are two cases that require special care. In the $5$-dimensional case, he requires the restriction of the map to the boundary to be a homeomorphism (rather than just a cell-like map). This is due to the fact that, at the time \cite{sieb} was written, it was unclear whether or not cell-like maps of (closed) $4$-manifolds could be approximated by homeomorphisms---hence the need of a stronger hypothesis on the boundary map. In view of Quinn's subsequent proof of the $4$-dimensional case \cite{quinn}, this stronger hypothesis is no longer needed in the $5$-dimensional boundary case. Note that, in our context, the bonding maps, when restricted to the boundary, are always cell-like (but are not homeomorphisms). 

The other special case has to do with $3$-dimensions. Here there is an added hypothesis that every point pre-image has a neighborhood $N$ which isn't just contractible, but in addition is prime (i.e. if $N= M_1 \# M_2$, then one of the $M_i$ is a 
standard $3$-sphere). The only way this could fail is if one of the $M_i$ were instead a homotopy $3$-sphere -- but by Perelman's resolution of the Poincar\'e Conjecture, such a manifold is automatically $S^3$. So again, in the $3$-dimensional case, this additional hypothesis is now unnecessary.
\end{rmk}

\section{Remarks on CAT(0) groups}\label{sec:CAT0}

In this section we remark on generalizing the main result from hyperbolic groups to CAT(0) groups.
A proper geodesic space $X$ is called \emph{CAT(0)} if geodesic triangles in $X$ are at least as thin as triangles in Euclidean space \cite{bridson+haefliger}. A group $G$ is called \emph{CAT(0)} if there exists a CAT(0) space $X$ on which $G$ acts \emph{geometrically} (that is, isometrically, properly, and compactly). 

A CAT(0) space $X$ has a visual boundary $\pa_\infty X$, and if $G$ acts geometrically on $X$, then $G$ acts on $\pa_\infty X$ by homeomorphisms. In this case $\pa_\infty X$ is called \emph{a} boundary of $G$. With this terminology we have the following theorem. 

\begin{thm}\label{thm:cat0sphere}
Let $G$ be a CAT(0) group for which $S^{n-1}$ is a boundary. If $n\ge6$, then there exists a closed $n$-dimensional aspherical manifold $W$ such that $\pi_1(W)\simeq G$.
\end{thm}

The proof is almost identical to the proof of Theorem \ref{thm:blw} in \cite{blw}. We give a short explanation for how to extend that argument to the CAT(0) case. 

\begin{proof}[Proof of Theorem \ref{thm:cat0sphere}] By assumption $G$ acts geometrically on an $X$ with $\pa_\infty X=S^{n-1}$. Denote $\ov X=X\cup\pa_\infty X$. We proceed in three steps. 

{\it Step 1.} $BG$ is homotopy equivalent to a closed aspherical homology $n$-manifold $W$ such that $W$ has the disjoint disk property. To show this, it suffices to show that $G$ is a PD$(n)$ group and to note that CAT(0) groups satisfy the Farrell-Jones conjectures in $K$- and $L$-theory. For then we may use \cite[Theorem 1.2]{blw}, which says that for such a group, $BG$ is homotopy equivalent to a closed aspherical homology $n$-manifold $M$ with the disjoint disk property. 

We explain why $G$ is PD$(n)$ group. First, we know $G$ is of type FP once we know that there exists a finite CW complex $K\simeq BG$ (for then the cellular chain complex of the universal cover $\wtil K$ is a finite length resolution of $\z$ by finitely generated free $G$ modules). A finite CW complex $K\simeq BG$ for a group $G$ that acts geometrically on a proper CAT(0) space is shown to exist by L\"uck \cite{luck_cat0model}. Now $G$ is a $PD(n)$ group because 
\[H^i(G;\z G)\cong H^i_c(X)\cong \wtil H^{i-1}(\pa_\infty X)=\wtil H^{i-1}(S^{n-1})=\left\{\begin{array}{ll}\z&\text{if }i=n\\0&\text{else}\end{array}\right.\]
The first two isomorphisms are described by Bestvina \cite{bestvina_localhomology}. That this implies $G$ is a PD($n)$ group is explained in \cite[VIII.10.1]{brown}. 

{\it Step 2.} The universal cover $\wtil W$ can be compactified $N=\wtil W\cup\pa_\infty X$ such that $N$ is a homology manifold with boundary. To show that $N$ is a homology manifold with boundary it suffices to show that $N$ is a finite-dimensional locally compact ANR and $\pa_\infty X$ is a $Z$-set in $N$ (see \cite[Proposition 2.5]{blw}). The pair $(\ov X,\pa_\infty X)$ is a $Z$-structure on $G$ by Bestvina \cite[Example 1.2(ii)]{bestvina_localhomology}. Furthermore, by \cite[Lemma 1.4]{bestvina_localhomology} for any other finite model $K$ for $BG$, there is a natural $Z$-structure on $(\ov K,\pa_\infty X)$, where $\ov K= K\cup\pa_\infty X$. Thus $(N,\pa_\infty X)$ admits a $Z$-set structure; in particular, $N$ is a Euclidean retract, finite dimensional, and $S^{n-1}$ is a $Z$-set inside $N$. 

{\it Step 3.} $\wtil W$ (and hence also $W$) is a manifold. This part of the argument is identical to that given in \cite[Theorem A]{blw}.  Quinn's invariant allows one to recognize manifolds among homology manifolds with the disjoint disk property. By the local nature of Quinn's invariant, if $(B,\pa B)$ is a homology manifold with boundary and $\pa B$ is a manifold, then int$(B)$ is a manifold. 
\end{proof}

In light of this result and Theorem \ref{thm:main} above, it is natural to ask the following question.

\begin{qu*}
Let $G$ be a CAT(0) group which admits $\scr S^{n-2}$ as a boundary. Is $G$ the fundamental group of an $n$-dimensional aspherical manifold with boundary? 
\end{qu*}

Examples of $G$ satisfying the hypothesis of this Question are given by Ruane \cite{ruane}: every nonuniform lattice $\Ga\le\SO(n,1)$ is an example. For these examples, an aspherical manifold with boundary can be obtained by ``truncating the cusps'' of $\mathbb H^n/\Ga$.

There are some basic problems with answering this Question with the techniques of this paper. For example, it is not obvious that peripheral subgroups of a CAT(0) group with Sierpinski space boundary are CAT(0), or that the double of a CAT(0) group along peripheral subgroups is CAT(0). 

\bibliographystyle{plain}
\bibliography{tpbib}

\end{document}